\documentclass[12pt]{article}

\usepackage[latin1]{inputenc}
\usepackage{amssymb,amsmath,amsthm}
\usepackage{amsfonts}
\usepackage{paralist}
\usepackage{booktabs}
\usepackage{ifthen}

\usepackage{color}
\usepackage{graphicx}

\usepackage{url}
\usepackage{hyphenat}

\newcommand{\vect}[1]{\ensuremath{\mathbf{#1}}}
\newcommand{\card}[1]{\ensuremath{\lvert{#1}\rvert}}

\newcommand{\IN}{\ensuremath{\mathbb{N}}}
\newcommand{\allpermutations}{\ensuremath{\mathbb{P}}}

\newcommand{\nset}[1]{\ensuremath{[{#1}]}}
\newcommand{\nsetSnl}{\ensuremath{\nset{n}^{\ell}_{\neq}}}

\newcommand{\subperm}[2]{\ensuremath{{#1}_{#2}}}
\newcommand{\substring}[2]{\ensuremath{{#1}[{#2}]}}

\newcommand{\symm}[1]{\ensuremath{S_{#1}}}         

\newcommand{\gensg}[1]{\ensuremath{\langle {#1} \rangle}}
\newcommand{\asc}[1]{\ensuremath{\iota_{#1}}}      
\newcommand{\desc}[1]{\ensuremath{\delta_{#1}}}    
\newcommand{\natcycle}[1]{\ensuremath{\zeta_{#1}}} 

\newcommand{\patt}[2]{\ensuremath{\Patl[{#1}]{#2}}}
\newcommand{\nsubl}[2]{\ensuremath{{\mathcal P}_{#2}{(#1)}}} 
\newcommand{\Snl}[1][\ell]{\ensuremath{\nsubl{n}{#1}}}
\newcommand{\groupasrel}[1]{\ensuremath{\gamma_{#1}}}
\newcommand{\Pow}{\mathcal{P}} 
\DeclareMathOperator{\range}{Im}
\DeclareMathOperator{\Sub}{Sub}                  
\DeclareMathOperator{\Inv}{Inv}                  
\DeclareMathOperator{\red}{red}                  
\DeclareMathOperator{\Pat}{Pat}
\DeclareMathOperator{\Comp}{Comp}
\DeclareMathOperator{\gPat}{gPat}
\DeclareMathOperator{\gComp}{gComp}
\DeclareMathOperator{\Av}{Av}
\DeclareMathOperator{\Aut}{Aut}
\DeclareMathOperator{\Rel}{Rel}
\DeclareMathOperator{\pcInv}{pcInv}
\DeclareMathOperator{\pcExt}{pcExt}

\newcommand{\Compn}[1][n]{\Comp^{(#1)}}
\newcommand{\Patl}[1][\ell]{\Pat^{(#1)}}
\newcommand{\gCompn}[1][n]{\gComp^{(#1)}}
\newcommand{\gPatl}[1][\ell]{\gPat^{(#1)}}

\theoremstyle{plain}
\newtheorem{theorem}{Theorem}[section]
\newtheorem{proposition}[theorem]{Proposition}
\newtheorem{lemma}[theorem]{Lemma}
\newtheorem{corollary}[theorem]{Corollary}

\theoremstyle{definition}
\newtheorem{definition}[theorem]{Definition}
\newtheorem{example}[theorem]{Example}

\newtheorem{remark}[theorem]{Remark}

\numberwithin{equation}{section}

\let\phi=\varphi
\let\rho=\varrho
\DeclareMathOperator{\preserves}{\triangleright}

\newcommand{\CHECK}[1]{{#1^{\scriptscriptstyle \lor}}}
\newcommand{\HAT}[1]{{#1^{\scriptscriptstyle \land}}}
\newcommand{\CHECKHAT}[1]{{#1^{\scriptscriptstyle \land\lor}}}
\newcommand{\HATCHECK}[1]{{#1^{\scriptscriptstyle \lor\land}}}

\title{Permutation groups, pattern involvement, \\ and Galois connections}

\author{Erkko Lehtonen
and
Reinhard P\"oschel \\[3ex]
Technische Universit\"at Dresden,
Institut f\"ur Algebra \\
01062 Dresden,
Germany}

\date{}

\begin{document}

\maketitle


\begin{abstract}
There is a connection between permutation groups and permutation patterns: for any subgroup $G$ of the symmetric group $\symm{\ell}$ and for any $n \geq \ell$, the set of $n$\hyp{}permutations involving only members of $G$ as $\ell$\hyp{}patterns is a subgroup of $\symm{n}$.
Making use of the monotone Galois connection induced by the pattern avoidance relation, we characterize the permutation groups that arise via pattern avoidance as automorphism groups of relations of a certain special form.
We also investigate a related monotone Galois connection for permutation groups and describe its closed sets and kernels as automorphism groups of relations.
\end{abstract}

\section{Introduction}\label{sect:intro}

The theory of pattern\hyp{}avoiding permutations has been an active field of research over the past decades.
For any permutations $\tau = \tau_1 \tau_2 \dots \tau_\ell \in \symm{\ell}$ and $\pi = \pi_1 \pi_2 \dots \pi_n \in \symm{n}$ ($\ell \leq n$), the permutation $\pi$ is said to involve $\tau$, or $\tau$ is called a pattern of $\pi$, if there exists a substring $\pi_{i_1} \pi_{i_2} \dots \pi_{i_\ell}$ ($i_1 < i_2 < \dots < i_\ell$) that is order\hyp{}isomorphic to $\tau$. If $\pi$ does not involve $\tau$, then $\pi$ is said to avoid $\tau$.
For more background and a survey on permutation patterns, we refer the reader to the monograph by Kitaev~\cite{Kitaev}.

The theory of permutation groups is a classical field of algebra that needs no special introduction here; see, e.g., Dixon and Mortimer~\cite{DixMor}.
At first sight, permutation patterns do not seem to have much to do with permutation groups.
However, and perhaps surprisingly, there is a relevant connection.
Namely, every pattern of the composition of two permutations equals the composition of some patterns of the two permutations (see Lemma~\ref{lem:Ppitau-PpiPtau}).
This fact seems to have received limited attention, although it was reported in the 1999 paper by Atkinson and Beals~\cite{AtkBea}, which deals with classes of permutations closed under pattern involvement and composition.
An important consequence of this simple yet crucial observation is that for any permutation group $G \leq \symm{n}$, every level of the class $\Av(\symm{n} \setminus G)$ of permutations avoiding the complement of $G$ is a permutation group.

This raises the question which permutation groups arise as sets of $n$\hyp{}permutations avoiding some sets of $\ell$\hyp{}permutations.
We will refer to such groups as $\ell$\hyp{}pattern subgroups of $\symm{n}$.
Our aim in this paper is to address this question by making use of Galois connections.
As for any binary relation, the pattern avoidance relation induces a monotone Galois connection -- referred to as $(\Patl, \Compn)$ -- between the sets $\symm{\ell}$ and $\symm{n}$ of permutations of the $\ell$\hyp{}element set and the $n$\hyp{}element set, respectively.
As an answer to the question, we would like to describe the closed sets of this monotone Galois connection, as well as those of its modification $(\gPatl, \gCompn)$ for permutation groups.

This paper is organized as follows.
In Section~\ref{sect:prelim}, we introduce the necessary basic definitions related to permutations and patterns.
In Section~\ref{sec:3}, we introduce the monotone Galois connections $(\Patl, \Compn)$ between the subsets of $\symm{\ell}$ and $\symm{n}$ and $(\gPatl, \gCompn)$ between the subgroups of $\symm{\ell}$ and $\symm{n}$, and we establish some of their basic properties.
In Section~\ref{sec:4}, we describe the $\ell$\hyp{}pattern subgroups of $\symm{n}$ as automorphism groups of relations.
We obtain a simpler description for the special class of $\ell$\hyp{}pattern subgroups of $\symm{n}$ that are of the form $\Compn G$ for some subgroup $G \leq \symm{\ell}$.
This is presented in Section~\ref{sec:5}, in which we also describe the Galois closures and kernels of $(\gPatl, \gCompn)$ as automorphism groups of relations.
Finally, in Section~\ref{sec:conclusion}, we make some concluding remarks and indicate possible directions for further research.


\section{Preliminaries}\label{sect:prelim}

\subsection*{General notation}

The letter $\IN$ stands for the set of nonnegative integers, and $\IN_+ := \IN \setminus \{0\}$.
For any $n \in \IN_+$, let $\nset{n} := \{1, \dots, n\}$.
The power set of a set $A$ is denoted by $\Pow(A)$, and the set of all $\ell$\hyp{}element subsets of $\nset{n}$ is denoted by $\Snl$.
Let $A^n_{\neq}$ be the set of all $n$\hyp{}tuples on a set $A$ with pairwise distinct entries.

  We will always compose functions from right to left, and we often
  denote functional composition simply by juxtaposition. Thus $fg(x) =
  (f \circ g)(x) = f(g(x))$.

  Since an $n$\hyp{}tuple $\vect{a} = (a_1, \dots, a_n) \in A^n$ is
  formally a map $\nset{n} \to A$, we can readily consider the image
  $\range\vect{a}=\{a_{1},\dots,a_{n}\}$ 
and form compositions of
  tuples with other functions. In particular, for any $\sigma \in
  \symm{n}$, we have $\vect{a} \sigma = (a_{\sigma(1)}, \dots,
  a_{\sigma(n)})$, and for any $\phi \colon A \to B$, we have $\phi
  \vect{a} = (\phi(a_1), \dots, \phi(a_n))$. We often write an
  $n$\hyp{}tuple $(a_1, \dots, a_n)$ as a string $a_1 \dots a_n$. If $I
  \subseteq \nset{n}$, $I = \{i_1, i_2, \dots, i_\ell\}$ with $i_1 <
  i_2 < \dots < i_\ell$, then we denote the substring
  $a_{i_1} a_{i_2} \dots a_{i_\ell}$ of $\vect{a}$ by
  $\substring{\vect{a}}{I}$.

\subsection*{Permutations}

  The symmetric group on the set
  $\nset{n}$ is denoted by $\symm{n}$. It comprises all permutations
  of the set $\nset{n}$, and the group operation is composition of
  functions. We will refer to permutations of $\nset{n}$ also as
  \emph{$n$\hyp{}permutations} when we want to emphasize the number
  $n$. The subgroup generated by a subset $S \subseteq \symm{n}$ is
  denoted by $\gensg{S}$.
We write $G \leq \symm{n}$ to express the fact that $G$ is a subgroup of $\symm{n}$.
For the sake of linguistic convenience, when we speak of a permutation group, we usually really mean the universe of a permutation group, i.e., the set of elements of the group stripped of its operation.

  Any permutation $\pi \in \symm{n}$ corresponds to an $n$\hyp{}tuple
  $(\pi_1, \pi_2, \dots, \pi_n)$ or, equivalently, to a string $\pi_1
  \pi_2 \dots \pi_n$, where $\pi_i = \pi(i)$ for all $i \in
  \nset{n}$. On the other hand, any $n$\hyp{}tuple $\vect{a} = (a_1,
  \dots, a_n) \in \nset{n}^n$ with no repeated elements, when viewed
  as a map $\nset{n} \to \nset{n}$, is a permutation of $\nset{n}$. In
  this paper, we will utilize these manifestations of
  permutations (bijective maps, tuples or strings), and we may switch
  between them at will, depending on which one we find the most
  convenient in each situation. In Figure~\ref{fig:1} we illustrate
  this with the graphical representation of the permutation
  $\pi=31524\in [5]^{5}_{\neq}$ 
  (the circles (shaded or not) represent the graph of the
  mapping $\pi\in \symm{5}$).

  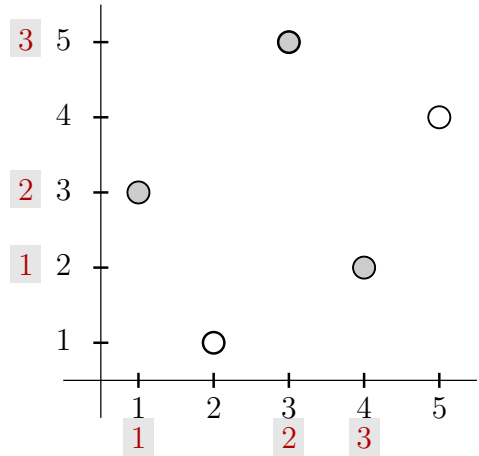
\begin{figure}
    \centering
\begin{picture}(0,0)%
\includegraphics{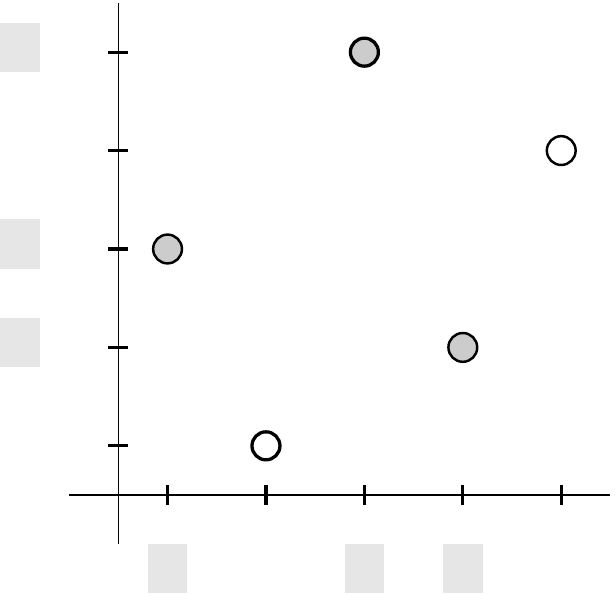}%
\end{picture}%
\setlength{\unitlength}{4144sp}%
\begingroup\makeatletter\ifx\SetFigFont\undefined%
\gdef\SetFigFont#1#2#3#4#5{%
  \reset@font\fontsize{#1}{#2pt}%
  \fontfamily{#3}\fontseries{#4}\fontshape{#5}%
  \selectfont}%
\fi\endgroup%
\begin{picture}(2803,2713)(1260,-2312)
\put(2881,-2266){\makebox(0,0)[lb]{\smash{{\SetFigFont{12}{14.4}{\rmdefault}{\mddefault}{\updefault}{\color[rgb]{.69,0,0}2}}}}}
\put(3331,-2266){\makebox(0,0)[lb]{\smash{{\SetFigFont{12}{14.4}{\rmdefault}{\mddefault}{\updefault}{\color[rgb]{.69,0,0}3}}}}}
\put(1306,119){\makebox(0,0)[lb]{\smash{{\SetFigFont{12}{14.4}{\rmdefault}{\mddefault}{\updefault}{\color[rgb]{.69,0,0}3}}}}}
\put(1306,-781){\makebox(0,0)[lb]{\smash{{\SetFigFont{12}{14.4}{\rmdefault}{\mddefault}{\updefault}{\color[rgb]{.69,0,0}2}}}}}
\put(1306,-1231){\makebox(0,0)[lb]{\smash{{\SetFigFont{12}{14.4}{\rmdefault}{\mddefault}{\updefault}{\color[rgb]{.69,0,0}1}}}}}
\put(1981,-2266){\makebox(0,0)[lb]{\smash{{\SetFigFont{12}{14.4}{\rmdefault}{\mddefault}{\updefault}{\color[rgb]{.69,0,0}1}}}}}
\put(1981,-2086){\makebox(0,0)[lb]{\smash{{\SetFigFont{12}{14.4}{\rmdefault}{\mddefault}{\updefault}{\color[rgb]{0,0,0}1}}}}}
\put(2431,-2086){\makebox(0,0)[lb]{\smash{{\SetFigFont{12}{14.4}{\rmdefault}{\mddefault}{\updefault}{\color[rgb]{0,0,0}2}}}}}
\put(2881,-2086){\makebox(0,0)[lb]{\smash{{\SetFigFont{12}{14.4}{\rmdefault}{\mddefault}{\updefault}{\color[rgb]{0,0,0}3}}}}}
\put(3331,-2086){\makebox(0,0)[lb]{\smash{{\SetFigFont{12}{14.4}{\rmdefault}{\mddefault}{\updefault}{\color[rgb]{0,0,0}4}}}}}
\put(3781,-2086){\makebox(0,0)[lb]{\smash{{\SetFigFont{12}{14.4}{\rmdefault}{\mddefault}{\updefault}{\color[rgb]{0,0,0}5}}}}}
\put(1531,-1681){\makebox(0,0)[lb]{\smash{{\SetFigFont{12}{14.4}{\rmdefault}{\mddefault}{\updefault}{\color[rgb]{0,0,0}1}}}}}
\put(1531,119){\makebox(0,0)[lb]{\smash{{\SetFigFont{12}{14.4}{\rmdefault}{\mddefault}{\updefault}{\color[rgb]{0,0,0}5}}}}}
\put(1531,-331){\makebox(0,0)[lb]{\smash{{\SetFigFont{12}{14.4}{\rmdefault}{\mddefault}{\updefault}{\color[rgb]{0,0,0}4}}}}}
\put(1531,-781){\makebox(0,0)[lb]{\smash{{\SetFigFont{12}{14.4}{\rmdefault}{\mddefault}{\updefault}{\color[rgb]{0,0,0}3}}}}}
\put(1531,-1231){\makebox(0,0)[lb]{\smash{{\SetFigFont{12}{14.4}{\rmdefault}{\mddefault}{\updefault}{\color[rgb]{0,0,0}2}}}}}
\end{picture}

    \caption{The permutation $\pi=31524$ and its pattern $\pi_{\{1,3,4\}}=231$}
    \label{fig:1}
  \end{figure}

  We will also use the conventional cycle notation for permutations:
  if $a_1, a_2, \dots, a_\ell$ are distinct elements of $\nset{n}$,
  then $(a_1 \; a_2 \; \cdots \; a_\ell)$ denotes the permutation that
  maps $a_\ell$ to $a_1$ and $a_i$ to $a_{i+1}$ for $1 \leq i \leq
  \ell - 1$ and keeps the remaining elements fixed. Such a permutation
  is called a \emph{cycle}, or an \emph{$\ell$\hyp{}cycle.} Every
  permutation is a product of pairwise disjoint cycles. For the
  permutation in Figure~\ref{fig:1} we have $\pi=(1\,3\,5\,4\,2)$.

  The following permutations will appear frequently in what follows:
  \begin{itemize}
  \item the \emph{identity permutation,} or the \emph{ascending
      permutation} $\asc{n} := 12 \dots n$,
  \item the \emph{descending permutation} $\desc{n} := n (n-1) \dots
    1$,
  \item the \emph{natural cycle} $\natcycle{n} := 2 3 \dots n 1 = (1
    \; 2 \; \cdots \; n)$.
  \end{itemize}

\subsection*{Permutation patterns and functional composition}

  First we recall some standard terminology from the theory of
  permutation patterns (see, e.g., B\'ona~\cite{Bona} or
  Kitaev~\cite{Kitaev}). For any string $\vect{u} = u_1 u_2 \dots u_m$ of distinct
  integers, the \emph{reduction} or \emph{reduced form} of $\vect{u}$,
  denoted by $\red(\vect{u})$, is the permutation obtained from
  $\vect{u}$ by replacing its $i$\hyp{}th smallest entry with $i$, for
  $1 \leq i \leq m$. A permutation $\tau \in
  \symm{\ell}$ is a \emph{pattern} (or an \emph{$\ell$\hyp{}pattern,}
  if we want to emphasize the number $\ell$) of a permutation $\pi \in
  \symm{n}$, or $\pi$ \emph{involves} $\tau$, denoted $\tau \leq \pi$,
  if there exists a substring $\vect{u} = \substring{\pi}{I} = \pi_{i_1} \pi_{i_2}
  \dots \pi_{i_\ell}$ of $\pi = \pi_1 \dots \pi_n$ (where $I = \{i_1,
  i_2, \dots, i_\ell\}$, $i_1 < i_2 < \dots < i_\ell$) such that
  $\red(\vect{u}) = \tau$. 
If $\tau \nleq \pi$, the permutation $\pi$ is said to \emph{avoid} $\tau$.

\begin{example}
Let $\pi = 31524 \in \symm{5}$ and $I = \{1,3,4\}$. Then
$\vect{u}=\substring{\pi}{I}=352$ and the corresponding pattern
is $\red(\vect{u})=231\in \symm{3}$ (denoted by $\pi_{I}$ using
Definition~\ref{def:hS}, cf.\ also the shaded circles in Figure~\ref{fig:1}).
\end{example}

We denote by $\patt{\ell}{\pi}$ the set of all $\ell$\hyp{}patterns of $\pi$, i.e., $\patt{\ell}{\pi} := \{\tau \in \symm{\ell} \mid \tau \leq \pi\}$.

\begin{example}
  \label{ex:asc-desc-cycle}
  For any $\ell, n \in \IN_+$ with $\ell < n$,
  \[
  \patt{\ell}{\asc{n}} = \{\asc{\ell}\}, \qquad \patt{\ell}{\desc{n}}
  = \{\desc{\ell}\}, \qquad \patt{\ell}{\natcycle{n}} = \{\asc{\ell},
  \natcycle{\ell}\}.
  \]
\end{example}

  The pattern involvement relation $\leq$ is a partial order on the
  set $\allpermutations := \bigcup_{n \geq 1} \symm{n}$ of all finite
  permutations. Downward closed subsets of $\allpermutations$ under
  this order are called \emph{permutation classes.} 
  For a permutation
  class $C$ and for $n \in \IN_+$, the set $C^{(n)} := C \cap
  \symm{n}$ is called the \emph{$n$\hyp{}th level} of $C$. For
  any set $B \subseteq \allpermutations$, let $\Av(B)$ be the set of
  all permutations that avoid every member of $B$. It is clear that
  $\Av(B)$ is a permutation class and, conversely, every permutation
  class is of the form $\Av(B)$ for some $B \subseteq
  \allpermutations$.

The notions introduced in the previous paragraphs can be expressed in
terms of order\hyp{}isomorphisms and functional composition as
follows.

\begin{definition}
  \label{def:hS}
  For any $I \in \Snl$, let $h_I \colon \nset{\ell} \to I$ be the
  order\hyp{}isomorphism $(\nset{\ell}, {\leq}) \to (I, {\leq})$. For
  $\pi \in \symm{n}$, define the permutation $\subperm{\pi}{I} \colon
  \nset{\ell} \to \nset{\ell}$ as $\subperm{\pi}{I} = h^{-1}_{\pi(I)}
  \circ \pi \circ h_I$ (here $\pi(I):=\{\pi(a) \mid a \in I\}$ and to be
  precise, $\subperm{\pi}{I} = h^{-1}_{\pi(I)}
  \circ \pi|_I \circ h_I$). 
\end{definition}

As explained in the beginning of this section, we can consider the mapping $h_I$ also as an $\ell$\hyp{}tuple in $I^\ell \subseteq \nset{n}^\ell$, namely, as the tuple consisting of the elements of $I$ in increasing order, or, using the notation for substrings, $h_I = \substring{\asc{n}}{I}$.

\begin{lemma}\label{A4}
  Let $n, \ell \in \IN_+$ with $\ell \leq n$.
  \begin{itemize}
    \item[\rm (i)] For any $\vect{a} \in \nset{n}^n$ and $I \in \Snl$, it holds that $\substring{\vect{a}}{I} = \vect{a} \circ h_I$.
    \item[\rm (ii)] For any $\vect{u} \in \IN^\ell_{\neq}$, it holds that $\red(\vect{u}) = h^{-1}_{\range \vect{u}} \circ \vect{u}$.
    \item[\rm (iii)] For all $\pi \in \symm{n}$, $\tau \in \symm{\ell}$, we have $\tau \leq \pi$ if and only if $\tau = \subperm{\pi}{I}$ for some $I \in \Snl$.
  \end{itemize}
\end{lemma}

\begin{proof}
  \begin{inparaenum}[(i)]
  \item Consider the composite function $\vect{a} \circ h_I$. The inner function $h_I$ provides the string $i_1 i_2 \dots i_\ell$ of indices, where $\{i_1, i_2, \dots, i_\ell\} = I$ and $i_1 < i_2 < \dots < i_\ell$. Subsequent application of $\vect{a}$ gives the substring $a_{i_1} a_{i_2} \dots a_{i_\ell} = \substring{\vect{a}}{I}$ of $\vect{a}$.
  
  \item Follows immediately from the definition of $\red(\vect{u})$ and $h_{\range \vect{u}}$.

  \item Being a composition of bijective maps, $\subperm{\pi}{I} = h_{\pi(I)}^{-1} \circ \pi \circ h_I$ is clearly a permutation of $\nset{\ell}$.
  By parts (i) and (ii), we have $h_{\pi(I)}^{-1} \circ \pi \circ h_I = \red(\substring{\pi}{I})$, so $\pi_I$ is indeed a pattern of $\pi$.
  Every $\ell$\hyp{}pattern of $\pi$ obviously arises in this way for some $I \in \Snl$.
  \end{inparaenum}
\end{proof}

In order to provide examples of the notions defined above and to help the reader feel comfortable with the formalism, we provide a simple, almost mechanical proof of the well\hyp{}known fact that pattern involvement is preserved under
reverses, complements, and inverses of permutations. Recall that the
\emph{reverse} of $\pi$ is $\pi^{\mathrm{r}} = \pi \circ \desc{n}$, and
the \emph{complement} of $\pi$ is $\pi^\mathrm{c} = \desc{n} \circ \pi$.
For example, $\pi = 31524$ we have $\pi^{\mathrm r} = 42513$ (the tuple
$\pi$ in reverse order) and
$\pi^{\mathrm c} = 35142$ (in each place we have the ``complement'' to
$n+1$: $6 = 3+3 = 1+5 = 5+1 = 2+4 = 4+2$).

\begin{lemma}
  \label{lem:rev-comp-inv}
  The following identities hold for any $\pi \in \symm{n}$ and $I \in
  \Snl$: 
  \begin{enumerate}[\upshape (i)]
  \item $\subperm{(\pi^\mathrm{r})}{I} =
    (\subperm{\pi}{\desc{n}(I)})^\mathrm{r}$,
  \item $\subperm{(\pi^\mathrm{c})}{I} =
    (\subperm{\pi}{I})^\mathrm{c}$,
  \item $\subperm{(\pi^{-1})}{I} = (\subperm{\pi}{\pi^{-1}(I)})^{-1}$.
  \end{enumerate}
\end{lemma}

\begin{proof}
The identities are verified by straightforward calculations:
  \begin{align*}
    \subperm{(\pi^\mathrm{r})}{I} &= h^{-1}_{\pi^\mathrm{r}(I)} \circ
    \pi^\mathrm{r} \circ h_I
    = h^{-1}_{\pi(\desc{n}(I))} \circ \pi \circ \desc{n} \circ h_I \\ 
    &= h^{-1}_{\pi(\desc{n}(I))} \circ \pi \circ h_{\desc{n}(I)} \circ
    h_{\desc{n}(I)}^{-1} \circ \desc{n} \circ h_I  \\
    &= \subperm{\pi}{\desc{n}(I)} \circ \subperm{(\desc{n})}{I} =
    \subperm{\pi}{\desc{n}(I)} \circ \desc{\ell}
    = (\subperm{\pi}{\desc{n}(I)})^\mathrm{r}, \\
    \subperm{(\pi^\mathrm{c})}{I} &= h^{-1}_{\pi^\mathrm{c}(I)} \circ
    \pi^\mathrm{c} \circ h_I
    = h^{-1}_{\desc{n}(\pi(I))} \circ \desc{n} \circ \pi \circ h_I \\
    &= h^{-1}_{\desc{n}(\pi(I))} \circ \desc{n} \circ h_{\pi(I)} \circ h_{\pi(I)}^{-1} \circ \pi \circ h_I \\
    &= \subperm{(\desc{n})}{\pi(I)} \circ \subperm{\pi}{I} =
    \desc{\ell} \circ \subperm{\pi}{I}
    = (\subperm{\pi}{I})^\mathrm{c}, \\
    \subperm{(\pi^{-1})}{I} &= h^{-1}_{\pi^{-1}(I)} \circ \pi^{-1}
    \circ h_I
    = (h_I^{-1} \circ \pi \circ h_{\pi^{-1}(I)})^{-1} \\
    &= (h_{\pi(\pi^{-1}(I))}^{-1} \circ \pi \circ
    h_{\pi^{-1}(I)})^{-1} =
    (\subperm{\pi}{\pi^{-1}(I)})^{-1}.
    \qedhere
  \end{align*}
\end{proof}

More interestingly, our formalism reveals that every
$\ell$\hyp{}pattern of a composition of permutations is a composition
of $\ell$\hyp{}patterns of the respective
permutations. (Lemma~\ref{lem:Ppitau-PpiPtau}\eqref{pitau} rephrases
\cite[Lemma~3]{AtkBea} in a slightly generalized way.)

\begin{lemma}
  \label{lem:Ppitau-PpiPtau}
  Let $\pi, \tau \in \symm{n}$, let $\ell \in \nset{n}$, and let $I \in \Snl$. Then the
  following statements hold.
  \begin{enumerate}[\rm (i)]
  \item\label{pitau} $\subperm{(\pi \tau)}{I} = \subperm{\pi}{\tau(I)}
    \circ \subperm{\tau}{I}$.
  \item\label{Ppitau} $\patt{\ell}{\pi \tau} \subseteq
    (\patt{\ell}{\pi}) (\patt{\ell}{\tau})$.
  \end{enumerate}
\end{lemma}

\begin{proof}
  \begin{inparaenum}[\rm (i)]
  \item The identity is verified by straightforward calculation:
    \[
    \subperm{(\pi \tau)}{I} = h_{(\pi \circ \tau)(I)}^{-1} \circ \pi
    \circ \tau \circ h_I = h_{\pi(\tau(I))}^{-1} \circ \pi \circ
    h_{\tau(I)} \circ h_{\tau(I)}^{-1} \circ \tau \circ h_I =
    \subperm{\pi}{\tau(I)} \circ \subperm{\tau}{I}.
    \]

  \item Let $\sigma \in \patt{\ell}{\pi \tau}$. Then there exists $I
    \in \Snl$ such that $\sigma = \subperm{(\pi \tau)}{I}$. By
    part~\eqref{pitau}, we have $\subperm{(\pi \tau)}{I} =
    \subperm{\pi}{\tau(I)} \circ \subperm{\tau}{I}$. Since
    $\subperm{\pi}{\tau(I)} \in \patt{\ell}{\pi}$ and
    $\subperm{\tau}{I} \in \patt{\ell}{\tau}$, we conclude that
    $\sigma \in (\patt{\ell}{\pi}) (\patt{\ell}{\tau})$.
  \end{inparaenum}
\end{proof}

\begin{remark}
  The converse inclusion $(\patt{\ell}{\pi}) (\patt{\ell}{\tau})
  \subseteq \patt{\ell}{\pi \tau}$ does not hold in general, and it is
  easy to find examples where we have a strict inclusion
  $\patt{\ell}{\pi \tau} \subsetneq (\patt{\ell}{\pi})
  (\patt{\ell}{\tau})$. For example, let $\pi = \tau = 132$ and $\ell
  = 2$. Then $\pi \tau = 123$ and $\patt{2}{\pi} = \patt{2}{\tau} =
  \{12, 21\}$, $\patt{2}{\pi \tau} = \{12\}$, $(\patt{2}{\pi})
  (\patt{2}{\tau}) = \{12, 21\}$.
\end{remark}

Using our formalism, it is also easy to prove the well\hyp{}known fact
that the pattern involvement relation is a partial order. Furthermore,
every covering relationship in this order links permutations of two
consecutive lengths.

\begin{lemma}
  \label{lem:trans-between}
  Assume that $\ell \leq m \leq n$.
  \begin{enumerate}[\rm (i)]
  \item If $I \in \nsubl{m}{\ell}$ and $J \in \nsubl{n}{m}$, then $h_J
    \circ h_I = h_{h_J(I)}$.
  \item For all $\sigma \in \symm{\ell}$, $\pi \in \symm{m}$, $\tau
    \in \symm{n}$, it holds that if $\sigma \leq \pi$ and $\pi \leq
    \tau$, then $\sigma \leq \tau$.
  \item For all $\sigma \in \symm{\ell}$, $\tau \in \symm{n}$, it
    holds that if $\sigma \leq \tau$, then there exists $\pi \in
    \symm{m}$ such that $\sigma \leq \pi \leq \tau$.
  \end{enumerate}
\end{lemma}

\begin{proof}
  \begin{inparaenum}[\rm (i)]
  \item For any $i \in \nset{\ell}$, the $i$\hyp{}th smallest element of the set $I$ is $h_I(i)$. The order\hyp{}isomorphism $h_J \colon \nset{m} \to J \subseteq \nset{n}$ maps the $i$\hyp{}th smallest element of $I$ to the $i$\hyp{}th smallest element of $h_J(I)$, which is $h_{h_J(I)}(i)$. Thus, $(h_J \circ h_I)(i) = h_J(h_I(i)) = h_{h_J(I)}(i)$.

  \item Our hypotheses assert that there exist $I \in \nsubl{m}{\ell}$
    and $J \in \nsubl{n}{m}$ such that $\sigma = \subperm{\pi}{I} =
    h_{\pi(I)}^{-1} \circ \pi \circ h_I$ and $\pi = \subperm{\tau}{J}
    = h_{\tau(J)}^{-1} \circ \tau \circ h_J$,
    where $h_I$ and $h_{\pi(I)}$ map $\nset{\ell}$ into $\nset{m}$, and $h_J$ and $h_{\tau(J)}$
    map $\nset{m}$ into $\nset{n}$.
    Observe that
    $h_{\tau(J)} \circ \pi = \tau \circ h_J$. Then, by part (i), we
    have
    \[
    h_{\pi(I)}^{-1} \circ h_{\tau(J)}^{-1} = (h_{\tau(J)} \circ
    h_{\pi(I)})^{-1} = (h_{h_{\tau(J)}(\pi(I))})^{-1} =
    (h_{\tau(h_J(I))})^{-1}.
    \]
    Consequently,
    \[
    \sigma = h_{\pi(I)}^{-1} \circ h_{\tau(J)}^{-1} \circ \tau \circ
    h_J \circ h_I = (h_{\tau(h_J(I))})^{-1} \circ \tau \circ
    h_{h_J(I)} = \subperm{\tau}{h_J(I)},
    \]
    which shows that $\sigma \leq \tau$.

  \item By the assumption that $\sigma \leq \tau$, there exists $I \in
    \Snl$ such that $\sigma = \subperm{\tau}{I} = h_{\tau(I)}^{-1}
    \circ \tau \circ h_I$. Let $J$ be any $m$\hyp{}element subset of
    $\nset{n}$ satisfying $I \subseteq J$, and let $\pi :=
    \subperm{\tau}{J}$. We have $\pi \leq \tau$ by definition, and it
    remains to show that $\sigma \leq \pi$. Consider
    \begin{align*}
      \subperm{\pi}{h_J^{-1}(I)} &= h_{\pi(h_J^{-1}(I))}^{-1} \circ
      \pi \circ h_{h_J^{-1}(I)}
      = h_{\subperm{\tau}{J}(h_J^{-1}(I))}^{-1} \circ \subperm{\tau}{J} \circ h_{h_J^{-1}(I)} \\
      &= h_{\subperm{\tau}{J}(h_J^{-1}(I))}^{-1} \circ
      h_{\tau(J)}^{-1} \circ \tau \circ h_J \circ h_{h_J^{-1}(I)}.
    \end{align*}
    By part (i), we have $h_J \circ h_{h_J^{-1}(I)} =
    h_{h_J(h_J^{-1}(I))} = h_I$, and
    \begin{align*}
      h_{\subperm{\tau}{J}(h_J^{-1}(I))}^{-1} \circ h_{\tau(J)}^{-1}
      &= (h_{\tau(J)} \circ h_{h_{\tau(J)}^{-1} \circ \tau \circ h_J (h_J^{-1}(I))})^{-1} \\
      &= (h_{h_{\tau(J)}(h_{\tau(J)}^{-1} \circ \tau \circ h_J
        (h_J^{-1}(I))})^{-1} = h_{\tau(I)}^{-1}.
    \end{align*}
    Thus, $\subperm{\pi}{h_J^{-1}(I)} = h_{\tau(I)}^{-1} \circ \tau
    \circ h_I = \subperm{\tau}{I} = \sigma$, which shows that $\sigma
    \leq \pi$.
  \end{inparaenum}
\end{proof}


\section{Pattern involvement and Galois \newline connections}
\label{sec:3} 

\subsection*{Galois connections}

A pair $(f,g)$ of maps $f:\Pow(A)\to\Pow(B)$, $g:\Pow(B)\to\Pow(A)$
between the power sets of sets $A$ and $B$ is called an \emph{(antitone) Galois
connection} or a \emph{monotone Galois connection}, respectively, if for all
$X \subseteq A$, $Y \subseteq B$ we
have
\begin{align*}
  X \subseteq g(Y) &\iff  f(X) \supseteq Y \text{ or}\\
  X \subseteq g(Y) &\iff  f(X) \subseteq Y \text{, respectively.}
\end{align*}
We collect some well-known facts. For a Galois connection the mappings $f$ and $g$ are antitone (order-reversing) and both compositions $f \circ g$ and $g \circ f$ are closure operators. 
For a monotone Galois connection the mappings $f$ and $g$ are monotone, $f$ and $g$ are called a \emph{lower} and an \emph{upper adjoint}, resp., and the composition $g \circ f$ is a closure operator while $f \circ g$ is a kernel operator.

Each binary relation $R \subseteq A \times B$ induces an antitone
Galois connection $(f,g)$ as
well as a monotone  Galois connection $(f^{*},g^{*})$ between the power set
lattices $\Pow(A)$ and $\Pow(B)$ via
\begin{align*}
  f(X) &:= \{b \in B \mid \forall a \in X \colon (a,b) \in R\},\\
  g(Y) &:= \{a \in A \mid \forall b \in Y \colon (a,b) \in R\},\\
  f^{*}(X) &:= A\setminus\{b \in B \mid \forall a \in X \colon (a,b) \in R\},\\
  g^{*}(Y) &:= \{a \in A \mid \forall b \in B\setminus Y \colon (a,b) \in R\},
\end{align*}
where $X \subseteq A$ and $Y \subseteq B$. Moreover, each Galois
connection between $\Pow(A)$ and $\Pow(B)$ is induced by a suitable
relation $R \subseteq A \times B$ (note $(a,b) \in R \iff b \in f(\{a\}) \iff b \notin f^{*}(\{a\})$).

Some further properties are mentioned in the following where
we consider a particular monotone Galois connection induced by the
pattern avoidance relation.
For general background and further information on Galois connections, we refer the reader to the book~\cite{DenErnWis}.

\subsection*{The operators $\Patl$ and $\Compn$}

Let $\ell, n \in \IN_+$ with $\ell \leq n$. We say that a permutation
$\tau \in \symm{n}$ is \emph{compatible} with a set $S \subseteq
\symm{\ell}$ of $\ell$\hyp{}permutations if $\patt{\ell}{\tau}
\subseteq S$. For $S \subseteq \symm{\ell}$, $T \subseteq \symm{n}$,
we write
\begin{align*}
  \Compn{S} &:= \{\tau \in \symm{n} \mid \patt{\ell}{\tau} \subseteq S\}, \\
  \Patl{T} &:= \bigcup_{\tau \in T} \patt{\ell}{\tau}.
\end{align*}
Thus, $\Compn{S}$ is the set of all $n$\hyp{}permutations compatible
with $S$, and $\Patl{T}$ is the set of all $\ell$\hyp{}patterns of
permutations in $T$. It is not difficult to verify that
\begin{align*}
  \Compn{S} &= \{\tau \in \symm{n} \mid \forall \sigma \in \symm{\ell} \setminus S \colon \sigma \nleq \tau\}, \\
  \Patl{T} &= \symm{\ell} \setminus \{\sigma \in \symm{\ell} \mid
  \forall \tau \in T \colon \sigma \nleq \tau\}.
\end{align*}
Consequently, $\Patl$ and $\Compn$ are precisely the lower and upper adjoints of the monotone Galois connection between
$\Pow(\symm{\ell})$ and $\Pow(\symm{n})$ induced by the
pattern avoidance relation $\nleq$.
Therefore,
\begin{equation}
\label{CompPatstar}
  \forall S \subseteq \symm{\ell} \,\, \forall T \subseteq \symm{n} \colon
    \Patl{T} \subseteq S \iff T \subseteq \Compn{S}.
\end{equation}
Furthermore, $\Patl{\Compn{}}$ and
$\Compn{\Patl{}}$ are kernel and closure operators, respectively.
The kernels and closures are just all the sets of the form $\Patl{T}$ and
$\Compn{S}$, respectively.
In particular we have
\[
\begin{array}{c@{\qquad}c}
  \Patl{\Compn{S}} \subseteq S, &
\Compn{\Patl{\Compn S}} =\Compn S, \\
  T \subseteq \Compn{\Patl{T}}, &
\Patl{\Compn{\Patl T}} = \Patl T.
\end{array}
\]

The upper adjoint $\Compn{}$ has the following remarkable property, on which the current work builds.

\begin{proposition}
  \label{prop:S-subgroup}
  If $S$ is a subgroup of $\symm{\ell}$, then $\Compn{S}$ is a
  subgroup of $\symm{n}$.
\end{proposition}

\begin{proof}
  Assume that $S \leq \symm{\ell}$.
Note that $\Compn{S}$ is nonempty, because $S$ is a group and hence contains the identity, from which it follows, by Example~\ref{ex:asc-desc-cycle}, that $\Compn{S}$ contains the identity.
  Let $\pi, \tau \in
  \Compn{S}$. Thus $\patt{\ell}{\pi}, \patt{\ell}{\tau} \subseteq
  S$. According to Lemmas~\ref{lem:rev-comp-inv}
  and~\ref{lem:Ppitau-PpiPtau}, we have
  \begin{gather*}
    \patt{\ell}{\pi^{-1}} = (\patt{\ell}{\pi})^{-1} := \{\sigma^{-1}
    \mid \sigma \in \patt{\ell}{\pi}\},
    \\
    \patt{\ell}{\pi \tau} \subseteq (\patt{\ell}{\pi})
    (\patt{\ell}{\tau}) = \{\sigma \sigma' \mid \sigma \in
    \patt{\ell}{\pi},\, \sigma' \in \patt{\ell}{\tau}\}.
  \end{gather*}
  Since $S$ is a group, it contains all products and inverses of its
  elements, so we have $\patt{\ell}{\pi^{-1}} \subseteq S$ and
  $(\patt{\ell}{\pi}) (\patt{\ell}{\tau}) \subseteq S$. Consequently,
  $\pi^{-1}, \pi \tau \in \Compn{S}$. This implies that $\Compn{S}$ is
  a subgroup of $\symm{n}$.
\end{proof}

Using the standard terminology of the theory of permutation patterns,
Proposition~\ref{prop:S-subgroup} can be rephrased as follows.

\begin{corollary}
  \label{cor:group-complement}
  The set of $n$\hyp{}permutations avoiding the complement of a
  subgroup of $\symm{\ell}$ is a subgroup of $\symm{n}$.
\end{corollary}

\subsection*{The operators $\gPatl$ and $\gCompn$}

In view of Proposition~\ref{prop:S-subgroup}, it makes sense to
modify the monotone Galois connection $(\Patl, \Compn)$ into one between
the subgroup lattices $\Sub(\symm{\ell})$ and $\Sub(\symm{n})$ of
the symmetric groups $\symm{\ell}$ and $\symm{n}$. Thus we define
$\gCompn \colon {\!\!}\linebreak \Sub(\symm{\ell}) \to \Sub(\symm{n})$, 
$\gPatl \colon \Sub(\symm{n}) \to \Sub(\symm{\ell})$ just by applying
the corresponding operators $\Compn$ and $\Patl$ and then taking the
generated subgroups. According to
 Proposition~\ref{prop:S-subgroup}, $\Compn{G} \in \Sub(\symm{n})$
 whenever $G \in \Sub(\symm{\ell})$, and we get:
\begin{align*}
  \gCompn{G} &:= \gensg{\Compn{G}} = \Compn{G}
  = \{\tau \in \symm{n} \mid \patt{\ell}{\tau} \subseteq G\}, \\
  \gPatl{H} &:= \gensg{\Patl{H}} = \gensg{\bigcup_{\tau \in H}
    \patt{\ell}{\tau}},
\end{align*}
where $G\in \Sub(\symm{\ell})$ and $H \in \Sub(\symm{n})$.

Let us verify that $({\gPatl}, {\gCompn})$ (or, equivalently,
$({\gPatl}\!, {\Compn})$) is indeed an adjoint pair of a monotone Galois
connection. 

\begin{lemma}\label{A5}
  For all $G \in \Sub(\symm{\ell})$ and $H \in \Sub(\symm{n})$, it
  holds that $\gPatl{H}\! \subseteq G$ if and only if $H \subseteq
  \gCompn{G}$.
\end{lemma}

\begin{proof}
First assume that $\gPatl{H} \subseteq G$. Then it holds that $\Patl H\subseteq G$
and we get $H\subseteq\Compn G=\gCompn G$ 
from equivalence~\eqref{CompPatstar}
and Proposition~\ref{prop:S-subgroup}. Assume then that  $H \subseteq 
\gCompn{G}=\Compn G$. Thus $\Patl H\subseteq G$ by
\eqref{CompPatstar} and 
therefore $\gPatl H= \gensg{\Patl H}\subseteq \gensg{G} = G$.
\end{proof}

The proof of Lemma~\ref{A5} remains valid if $H$ is an arbitrary 
subset of $\symm{n}$. Thus $(\gPatl, \gCompn)$ can be considered also as
a monotone Galois connection between $\Pow(\symm{n})$ and
$\Sub(\symm{\ell})$. In particular, the following lemma holds.

\begin{lemma}\label{lem:R1}
$\gPatl H = \gPatl \gensg{H}$
for $H \subseteq \symm{n}$.
\end{lemma}

\begin{proof}
Indeed, we have $\gPatl H \subseteq \gPatl \gensg{H}$ by the monotonicity of
\linebreak
$\gPatl$. Furthermore, because
$\gCompn\gPatl$ is the closure operator associated with the monotone
Galois connection $(\gPatl, \gCompn)$ between $\Pow(\symm{n})$ and
$\Sub(\symm{\ell})$, we have $H\subseteq\gCompn\gPatl H$, which implies
$\gensg H \subseteq \linebreak \gensg{\gCompn\gPatl H} = \gCompn\gPatl H$, 
and therefore $\gPatl\gensg{H} \subseteq \linebreak \gPatl\gCompn\gPatl H = \gPatl H$.
\end{proof}

\subsection*{$\ell$\hyp{}pattern subgroups}

Subgroups of $\symm{n}$ of the form $\Compn{S}$ for some subset $S
\subseteq \symm{\ell}$ are called \emph{$\ell$\hyp{}pattern subgroups}
of $\symm{n}$. As we have seen in Proposition~\ref{prop:S-subgroup},
$\Compn{S}$ is a subgroup of $\symm{n}$ whenever $S$ is a subgroup of
$\symm{\ell}$. On the other hand, it is well possible that $\Compn{S}$ is a group even if $S$
is not, and there are subgroups of $\symm{n}$ that are not
$\ell$\hyp{}pattern subgroups for any $\ell < n$.

\begin{example}
  For all $\ell, n \in \IN_+$ with $\ell < n$ and $n \geq 3$, the group
  $\gensg{(1 \; n)} \subseteq \symm{n}$ is not an $\ell$\hyp{}pattern
  subgroup of $\symm{n}$.
  The claim is obvious when $\ell \leq 2$, and it is easy to verify for $\ell = 3$ with the help of Proposition~\ref{prop:small-l} below.
  Assume that $\ell \geq 4$, and suppose, to the contrary, that $\gensg{(1 \;
    n)} = \Compn{S}$ for some $S \subseteq \symm{\ell}$. Then
  \[
  \Compn{\Patl{\gensg{(1 \; n)}}} = \Compn{\Patl{\Compn{S}}} =
  \Compn{S} = \gensg{(1 \; n)},
  \]
  that is, $\gensg{(1 \, n)}$ is a closed set. It is easy to verify
  that
  \begin{gather*}
    \Patl{\gensg{(1 \; n)}} = \{\asc{\ell}, (1 \; \ell), \natcycle{\ell}, \natcycle{\ell}^{-1}\}, \\
    \Compn{\{\asc{\ell}, (1 \; \ell), \natcycle{\ell},
      \natcycle{\ell}^{-1}\}} = \{\asc{n}, (1 \; n), \natcycle{n},
    \natcycle{n}^{-1}\}.
  \end{gather*}
  Thus $\gensg{(1 \; n)} \subsetneq \Compn{\Patl{\gensg{(1 \; n)}}}$,
  a contradiction.
\end{example}

These observations raise the questions which subgroups of $\symm{n}$
are $\ell$\hyp{}pattern subgroups and for which subsets $S \subseteq
\symm{\ell}$, the set $\Compn{S}$ is a group. For small values of
$\ell$, we can provide a conclusive answer.

\begin{proposition}
  \label{prop:small-l}
  Let $n, \ell \in \IN_+$ with $n \geq \ell$ and $\ell \leq 3$, and
  let $S$ be a subset of $\symm{\ell}$. Then $\Compn{S}$ is a subgroup
  of $\symm{n}$ if and only if $S$ is a subgroup of $\symm{\ell}$.
\end{proposition}

\begin{proof}
  The claim clearly holds for $S = \emptyset$, and we may assume that $S$ is nonempty.
  Sufficiency follows from Proposition~\ref{prop:S-subgroup}, and we
  only need to show necessity. The case when $\ell = 1$ is trivial:
  all nonempty subsets of $\symm{1}$ are subgroups of $\symm{1}$. It is easy to
  see that the claim holds when $\ell = 2$. The only nonempty subset of
  $\symm{2}$ that is not a subgroup is $\{21\}$, and $\Compn{\{21\}} =
  \{\desc{n}\}$ is not a subgroup of $\symm{n}$. It remains to
  consider the case when $\ell = 3$. Assume that $S \subseteq
  \symm{\ell}$ and $H := \Compn{S}$ is a subgroup of $\symm{n}$. Since
  $H$ is a group, it contains the identity permutation $\asc{n}$;
  therefore $S$ must contain all $3$\hyp{}patterns of $\asc{n}$, of
  which there is only one, namely $\asc{3} = 123$. If $\card{S} = 1$,
  then $S = \{123\}$, so $S$ is a group. Assume now that $\card{S}
  \geq 2$. We need to consider several cases.

  \begin{inparaenum}[\it {Case} 1.]
  \item\label{case:321} Assume that $321 \in S$. If $\card{S} = 2$,
    then $S = \{123, 321\} = \{\asc{3}, (1 \; 3)\}$, so $S$ is a subgroup of
    $\symm{\ell}$. Let us assume that $\card{S} \geq 3$. The
    descending permutation $\desc{n}$ is in $H$, because
    $\patt{3}{\desc{n}} = \{321\}$. Consequently, $H$ contains the
    reverses and complements of its members. Observe that the
    permutations
    \begin{align*}
      & (1 \; 2) = 213 \dots n, &
      & (1 \; 2)^\mathrm{r} = n (n-1) \dots 3 1 2, \\
      & (1 \; 2)^\mathrm{c} = (n-1) n (n-2) (n-3) \dots 1, & & (1 \;
      2)^\mathrm{rc} = 12 \dots (n-2) n (n-1)
    \end{align*}
    can be obtained from each other by taking reverses and
    complements, and
    \begin{align*}
      & \patt{3}{(1 \; 2)} = \{123, 213\}, &
      & \patt{3}{(1 \; 2)^\mathrm{r}} = \{321, 312\}, \\
      & \patt{3}{(1 \; 2)^\mathrm{c}} = \{321, 231\}, & & \patt{3}{(1
        \; 2)^\mathrm{rc}} = \{123, 132\}.
    \end{align*}
    Since $\{123, 321\} \subseteq S$ and $\{132, 213, 231, 312\} \cap
    S \neq \emptyset$, one of $\patt{3}{(1 \; 2)}$, $\patt{3}{(1 \;
      2)^\mathrm{r}}$, $\patt{3}{(1 \; 2)^\mathrm{c}}$, $\patt{3}{(1
      \; 2)^\mathrm{rc}}$ is included in $S$. Therefore $H$ contains
    one of $(1 \; 2)$, $(1 \; 2)^\mathrm{r}$, $(1 \; 2)^\mathrm{c}$,
    $(1 \; 2)^\mathrm{rc}$ and hence it contains all of them. It
    follows that $\symm{3} \subseteq \Patl[3]{H} \subseteq S \subseteq
    \symm{3}$, that is, $S = \symm{3}$, so $S$ is a subgroup of
    $\symm{3}$.

  \item\label{case:231-312} Assume that $\{231, 312\} \cap S \neq
    \emptyset$. Since
    \[
    \patt{3}{23 \dots n1} = \{123, 231\}, \qquad \patt{3}{n12 \dots
      (n-1)} = \{123, 312\},
    \]
    we have that $\natcycle{n} = 23 \dots n1 \in H$ or
    $\natcycle{n}^{-1} = n12 \dots (n-1) \in H$. Since $H$ is a group,
    it follows that $\gensg{\natcycle{n}} \subseteq H$; hence, in
    fact, both $\natcycle{n}$ and $\natcycle{n}^{-1}$ are in $H$, so
    $\{123, 231, 312\} \subseteq \Patl{H} \subseteq S$. If $S = \{123,
    231, 312\}$, then $S$ is a subgroup of $\symm{3}$. Let us assume
    that $\card{S} \geq 4$; then $\{132, 213, 321\} \cap S \neq
    \emptyset$. We have already dealt with the case when $321 \in S$
    (Case~\ref{case:321} above), so we may assume that $\{132, 213\}
    \cap S \neq \emptyset$. Since
    \[
    \patt{3}{(1 \; 2)} = \{123, 213\}, \quad \patt{3}{(n-1 \;\; n)} =
    \{123, 132\},
    \]
    $H$ contains $(1 \; 2)$ or $(n-1 \;\; n)$. Therefore $H$ includes
    $\{\natcycle{n}, (1 \; 2)\}$ or $\{\natcycle{n}, (n-1 \;\; n)\}$;
    these are generating sets of $\symm{n}$, so $H =
    \symm{n}$.
Consequently, $S = \symm{3}$.

  \item\label{case:132-213} Assume that $\{132, 213\} \cap S \neq
    \emptyset$. If $\card{S} = 2$, then $S = \{123, 132\}$ or $S =
    \{123, 213\}$, so $S$ is a subgroup of $\symm{3}$. Let us assume
    that $\card{S} \geq 3$. We have already dealt with the cases when
    $\{321, 231, 312\} \cap S \neq \emptyset$ (Cases~\ref{case:321},
    \ref{case:231-312} above), so we may assume that $\{132, 213\}
    \subseteq S$. Then all adjacent transpositions $(i \;\; i+1)$ ($1
    \leq i \leq n-1$) are in $H$, because $\patt{3}{(i \;\; i + 1)}
    \subseteq \{123, 132, 213\} \subseteq S$. These transpositions
    generate the symmetric group $\symm{n}$, so $H =
    \symm{n}$.
Consequently, $S = \symm{3}$.
  \end{inparaenum}

  The cases considered above exhaust all possibilities. This concludes
  the proof.
\end{proof}

It is easy to determine the groups $\Compn{G}$ for each subgroup $G$ of $\symm{3}$; these are
summarized in Table~\ref{table:symm3}.

\begin{table}
  \begin{center}
    \begin{tabular}{ll}
      \toprule
      $G$ & $\Compn{G}$ \\
      \midrule
      $\emptyset$ & $\emptyset$ \\
      $\{123\}$ & $\{12 \dots n\}$ \\
      $\{123, 132\}$ & $\gensg{(n-1 \; n)}$ \\
      $\{123, 213\}$ & $\gensg{(1 \; 2)}$ \\
      $\{123, 321\}$ & $\gensg{\desc{n}}$ \\
      $\{123, 231, 312\}$ & $\gensg{\natcycle{n}}$ \\
      $\symm{3}$ & $\symm{n}$ \\
      \bottomrule
    \end{tabular}
  \end{center}

  \bigskip
  \caption{Permutation groups compatible with subgroups of $\symm{3}$.}
  \label{table:symm3}
\end{table}

In Proposition~\ref{prop:small-l}, the hypothesis $\ell \leq 3$ is
indispensable. In fact, as the following lemma and examples
illustrate, counterexamples can be found when $\ell \geq 4$.

\begin{lemma}
  Assume that $\ell \geq 4$ and $n = \ell + 1$. Then there exists a
  group $H \leq \symm{n}$ such that $H = \Compn{S}$ for some subset $S
  \subseteq \symm{\ell}$, but there is no subgroup $G \leq
  \symm{\ell}$ such that $H = \Compn{G}$.
\end{lemma}

\begin{proof}
  Let $H := \gensg{(1 \; n)^\mathrm{r}} = \{12 \dots n, 1 (n-1) (n-2)
  \dots 2 n\}$ and $S := \Patl{H}$. Then $S$ comprises the following
  four permutations:
  \begin{itemize}
  \item $12 \dots (n-1) = \asc{n-1}$,
  \item $1 (n-1) (n-2) \dots 2 = (1 \; 2 \; \cdots \; n-1)^\mathrm{r}
    = [(1 \; 2 \; \cdots \; n-1)^{-1}]^\mathrm{c}$,
  \item $(n-2) (n-3) \dots 1 (n-1) = [(1 \; 2 \; \cdots \;
    n-1)^{-1}]^\mathrm{r} = (1 \; 2 \; \cdots \; n-1)^\mathrm{c}$,
  \item $1 (n-2) (n-3) \dots 2 (n-1) = (1 \; n-1)^\mathrm{r} = (1 \;
    n-1)^\mathrm{c}$.
  \end{itemize}

  First, let us verify that $\Compn{S} = H$. Let $\pi \in
  \Compn{S}$. Then $\patt{n-1}{\pi} \subseteq S$, by
  definition. Consider first the case when $12 \dots (n-1) \leq
  \pi$. Then $321 \nleq \pi$, so $1 (n-1) (n-2) \dots 2$ and $(n-2)
  (n-3) \dots 1 (n-1)$ cannot be $(n-1)$\hyp{}patterns of $\pi$,
  because $321$ is involved in both and pattern involvement is a transitive relation (see Lemma~\ref{lem:trans-between}).
  Similarly, if $n \geq 6$, then $1 (n-2) (n-3) \dots 2 (n-1)$ cannot be an
  $(n-1)$\hyp{}pattern of $\pi$. If $n = 5$, then it may be possible that
  both $1234$ and $1324$ are patterns of $\pi$; in this case $\pi$ is
  either $12435$ or $13245$. But we have $1243 \leq 12435$ and $2134
  \leq 13245$, yet $1243 \notin S$ and $2134 \notin S$. We conclude that the only
  $(n-1)$\hyp{}pattern of $\pi$ is $12 \dots (n-1)$, which implies
  that $\pi = 12 \dots n$.

  Consider then the case when $12 \dots (n-1) \nleq \pi$. Observe
  first that every permutation $\tau = \tau_1 \tau_2 \dots \tau_{n-1} \in S \setminus \{12 \dots
  (n-1)\}$ satisfies $\tau_2 > \tau_3$. Thus, from the fact that
  $\subperm{\pi}{\{1, \dots, n-1\}} \in S \setminus \{12 \dots
  (n-1)\}$, we get $\pi_2 > \pi_3$. At the same time, the only
  permutation $\tau \in S$ satisfying $\tau_1 > \tau_2$ is $(n-2)
  (n-3) \dots 1 (n-1)$. This implies that $\subperm{\pi}{\{2, \dots,
    n\}} = (n-2) (n-3) \dots 1 (n-1)$, so $\pi_n > \pi_2 > \pi_3 >
  \dots > \pi_{n-1}$. The only permutation $\tau \in S$ satisfying
  $\tau_2 > \tau_{n-1}$ is $1 (n-1) (n-2) \dots 2$, which implies that
  $\subperm{\pi}{\{1, \dots, n\}} = 1 (n-1) (n-2) \dots 2$, so
  $\pi_{n-1} > \pi_1$. Putting together the above inequalities
  involving $\pi_i$'s, we get $\pi = 1 (n-1) (n-2) \dots 2 n$. We
  conclude that indeed $\Compn{S} = H$.

  We still need to show that there is no subgroup $G \leq \symm{\ell}$
  such that $\Compn{G} = H$. Suppose, to the contrary, that $\Compn{G}
  = H$ for a subgroup $G \leq \symm{\ell}$. Then $S = \Patl{H}
  \subseteq G$, so $\gensg{S} \subseteq \gensg{G} = G$.

  The group $\gensg{S}$ contains
  \begin{itemize}
  \item $(1 \; n-1)^\mathrm{r} \circ (1 \; 2 \; \cdots \;
    n-1)^\mathrm{c} = (1 \; n-1) \circ \desc{n-1} \circ \desc{n-1}
    \circ (1 \; 2 \; \cdots \; n-1) = (1 \; n-1) \circ (1 \; 2 \;
    \cdots \; n-1) = (1 \; 2 \; \cdots \; n-2)$,
  \item $(1 \; 2 \; \cdots \; n-1)^\mathrm{r} \circ (1 \;
    n-1)^\mathrm{c} = (1 \; 2 \; \cdots \; n-1) \circ \desc{n-1} \circ
    \desc{n-1} \circ (1 \; n-1) = (1 \; 2 \; \cdots \; n-1) \circ (1
    \; n-1) = (2 \; 3 \; \cdots \; n-1)$,
  \item $(2 \; 3 \; \cdots \; n-1)^{-1} \circ (1 \; 2 \; \cdots \;
    n-1) = (1 \; n-1)$,
  \item $(1 \; n-1) \circ (1 \; n-1)^\mathrm{r} = (1 \; n-1) \circ (1
    \; n-1) \circ \desc{n-1} = \desc{n-1}$,
  \item $(1 \; 2 \; \cdots \; n-1)^\mathrm{r} \circ \desc{n-1} = (1 \;
    2 \; \cdots \; n-1) \circ \desc{n-1} \circ \desc{n-1} = (1 \; 2 \;
    \cdots \; n-1)$.
  \end{itemize}
  Since $\gensg{S}$ contains $\{(1 \; n-1), (1 \; 2 \; \cdots \;
  n-1)\}$, a generating set of $\symm{n-1}$, we conclude that
  $\gensg{S} = \symm{n-1}$, so $G = \symm{n-1}$. But
  $\Compn{\symm{n-1}} = \symm{n} \neq H$, a contradiction.
\end{proof}

\begin{example}
  For $n = 6$, $\ell = 4$, the authors have verified with computer
  (using the GAP algebra system)
  that a subgroup $H \leq \symm{n}$ is of the form $\Compn{S}$ for
  some subset $S \subseteq \symm{\ell}$ but not of the form
  $\Compn{G}$ for any subgroup $G \leq \symm{\ell}$ if and only if $H$
  is one of the following:
 \begin{center}
$\gensg{(2 \; 5)}$,
\quad
$\gensg{(2 \; 4)(3 \; 5)}$,
\quad
$\gensg{(2 \; 5)(3 \; 6)}$,
\quad
$\gensg{(1 \; 3)(4 \; 6)}$,
\quad
$\gensg{(1 \; 4)(2 \; 5)}$,
\quad
$\gensg{(1 \; 5)(2 \; 6)}$,
\quad
$\gensg{(1 \; 4 \; 5)(2 \; 3 \; 6)}$,
\quad
$\gensg{(2 \; 4)(3 \; 5),(1 \; 6)(2 \; 3)(4 \; 5)}$.
  \end{center}
\end{example}

\begin{example}
  For $n = 7$, $\ell = 4$, examples of subgroups of $\symm{n}$ that
  are of the form $\Compn{S}$ for some subset $S \subseteq
  \symm{\ell}$ but not of the form $\Compn{G}$ for any subgroup $G
  \leq \symm{\ell}$ include the following:
  \begin{center}
   $\gensg{(2 \; 5)(3 \; 6)}$,
   \quad
   $\gensg{(2 \; 6)(3 \; 7)}$,
   \quad
   $\gensg{(1 \; 5)(2 \; 6)}$.
  \end{center}

  For $n = 8$, $\ell = 4$, examples of subgroups of $\symm{n}$ that
  are of the form $\Compn{S}$ for some subset $S \subseteq
  \symm{\ell}$ but not of the form $\Compn{G}$ for any subgroup $G
  \leq \symm{\ell}$ include the following:
  \begin{center}
  $\gensg{(2 \; 6)(3 \; 7)}$.
  \end{center}

  Note that the above lists may not be exhaustive.
\end{example}


\section{On the monotone Galois connection \newline $\boldsymbol{(\Pat,\Comp)}$}
\label{sec:4}

The $\ell$\hyp{}pattern subgroups of $\symm{n}$ are those subgroups that are of the form $\Compn S$ for some subset $S \subseteq \symm{\ell}$.
As a way of describing such subgroups, we make use of another, classical Galois connection, namely the Galois connection $(\Aut, \Inv)$ between permutations and relations.
Recall that a $k$\hyp{}ary relation on a set $A$ is simply
a subset of $A^k$. Denote by $\Rel_n^{(k)}$ the set of all
$k$\hyp{}ary relations on $\nset{n}$, and let $\Rel_n := \bigcup_{k
  \geq 1} \Rel_n^{(k)}$.

Let $\pi \in \symm{n}$, and let $\rho \in \Rel_n$. We say that the
permutation $\pi$ \emph{preserves} the relation $\rho$, or that $\rho$
is an \emph{invariant} of $\pi$, or that $\pi$ is an
\emph{automorphism} of $\rho$, and we write $\pi \preserves \rho$,
if $\pi(\vect{r}) \in \rho$ for every $\vect{r} \in \rho$. The
preservation relation induces the Galois connection $(\Aut, \Inv)$, where
\begin{align*}
  \Aut R &:= \{\pi \in \symm{n} \mid \text{$\pi \preserves \rho$ for all $\rho \in R$}\}, \\
  \Inv S &:= \{\rho \in \Rel_n \mid \text{$\pi \preserves \rho$ for all $\pi \in S$}\},
\end{align*}
for every $S \subseteq \symm{n}$ and $R \subseteq \Rel_n$.
Thus, $\Aut R$ is the automorphism group of $R$. Write $\Inv^{(\ell)}
S := \Inv S \cap \Rel_n^{(\ell)}$.

It is well known that finite permutation groups are precisely the Galois
closures of the Galois connection $(\Aut, \Inv)$ (see, e.g., Chapter~8
in~\cite{PosKal}). A permutation group $H\leq\symm{n}$ is
\emph{$\ell$\hyp{}closed}, if it is the automorphism group of its
$\ell$\hyp{}ary invariant relations, i.e., $H = \Aut \Inv^{(\ell)} H$.

Let $H \leq \symm{n}$, and let $\vect{a} = (a_1, \dots, a_\ell) \in
\nset{n}^\ell_{\neq}$. Let
\[
\vect{a}^H := \{\sigma(\vect{a}) \mid \sigma \in H\} = \{(\sigma(a_1),
\dots, \sigma(a_\ell)) \mid \sigma \in H\}.
\]
A set of the form $\vect{a}^H$ for some $\vect{a} \in
\nset{n}^\ell_{\neq}$ is called an \emph{$\ell$\hyp{}orbit} of
$H$. For $I \in \Snl$, recall the map $h_I \colon \nset{\ell} \to I$ from
Definition~\ref{def:hS} and view it as a tuple
$h_I \in \nsetSnl$. Therefore it makes sense to consider the
$\ell$\hyp{}orbit $(h_I)^{H}$.

For $m \in \IN_+$, any group $G \leq \symm{m}$ can be viewed as an $m$\hyp{}ary irreflexive relation on the set $\nset{m}$ (i.e., a subset of $\nset{m}^m_{\neq}$) whose members are the permutations of $G$ viewed as tuples.
We denote this relation by $\groupasrel{G}$.
(Formally $\groupasrel{G} = G$, but we prefer to introduce the notation $\groupasrel{G}$ in order to avoid the expression $\Aut G$, because automorphisms of a group $G$ have another fixed meaning.)
It holds that $\groupasrel{G} = (1, \dots, m)^G = (h_{\nset{m}})^G$, where $h_{\nset{m}} \colon \nset{m} \to \nset{m}$ is just the identity map on $\nset{m}$.
Furthermore, the equality $G = \Aut \groupasrel{G}$ holds.

\begin{proposition}
  \label{prop:l-closed}
  Let $H \leq \symm{n}$, and assume that $H = \Compn{S}$ for some
  subset $S \subseteq \symm{\ell}$ \textup{(}not necessarily a
  subgroup\textup{)}. Then $H$ is $\ell$\hyp{}closed, i.e., $H$ is
  determined by its $\ell$\hyp{}ary invariant relations:
  \[
  H = \Aut \Inv H = \Aut \Inv^{(\ell)} H.
  \]
  In particular, the $\ell$\hyp{}orbits $(h_I)^H$ are enough to
  characterize the group:
  \[
  H = \Aut \{(h_I)^H \mid I \in \Snl\}.
  \]
\end{proposition}

\begin{proof}
  Clearly, $H \subseteq \Aut (h_I)^H$ for each $I \in \Snl$. Thus, $H
  \subseteq H'$, where $H' := \Aut \{(h_I)^H \mid I \in \Snl\}$. In
  order to prove the converse inclusion $H' \subseteq H$, it suffices
  to show that $\patt{\ell}{\tau} \subseteq S$ for every $\tau \in
  H'$. Let $\tau \in H'$ and $I \in \Snl$, and consider the pattern
  $\subperm{\tau}{I}$. Since $h_I \in (h_I)^H$ and $\tau \in \Aut
  (h_I)^H$, we have $\tau \circ h_I \in (h_I)^H$, i.e., $\tau \circ
  h_I = \pi \circ h_I$ for some $\pi \in H$. Therefore $\tau(I) =
  \pi(I)$, and since all $\ell$\hyp{}patterns of $\pi$ are in $S$, we
  have $\subperm{\tau}{I} = h_{\tau(I)}^{-1} \circ \tau \circ h_I =
  h_{\pi(I)}^{-1} \circ \pi \circ h_I = \subperm{\pi}{I} \in S$.
\end{proof}

In the following theorem, we describe the $\ell$\hyp{}pattern
subgroups of $\symm{n}$ as automorphism groups of relations
of a certain prescribed form. This theorem provides
a description of all groups of the form $\Compn S$, where $S$ is an
arbitrary subset of $\symm{\ell}$.
Those $\ell$\hyp{}pattern subgroups that are of the form $\Compn G$ for some subgroup $G$ of $\symm{\ell}$ have an even simpler description, which we will discuss in the next section (Theorem~\ref{thm:CompPat}).

\begin{theorem}
  \label{thm:Comp-Aut-arbitrary-S}
  Let $H \leq \symm{n}$, and consider the $\ell$\hyp{}orbits $\rho_I
  := (h_I)^H$ for all $I \in \Snl$. Then $H$ is of the form $H =
  \Compn{S}$ for some $S \subseteq \symm{\ell}$ if and only if
  \begin{enumerate}[\rm (a)]
  \item\label{CAAS:a} $H = \Aut \{\rho_I \mid I \in \Snl\}$,
  \item\label{CAAS:b} the $\rho_I$ satisfy the following property: for
    every $x \in \nset{n}^n_{\neq}$ we have
    \begin{multline*}
    \bigl( \forall I \in \Snl \, \exists J \in \Snl \colon
    \red(\substring{x}{I}) \in \red(\rho_J) \bigr) 
    \\
    \implies \forall I
    \in \Snl \colon \substring{x}{I} \in \rho_I.
    \end{multline*}
  \end{enumerate}
\end{theorem}

\begin{proof}
  Concerning the implication in condition~\eqref{CAAS:b}, observe that
  the antecedent $\forall I \in \Snl \, \exists J \in \Snl \colon
  \red(\substring{x}{I}) \in \red(\rho_J)$ expresses the fact that
  each $\ell$\hyp{}pattern of $x$, now considered as a permutation $x \in \nset{n}^n_{\neq} = \symm{n}$,
  coincides with an $\ell$\hyp{}pattern of some
  permutation from $H$, i.e., $x \in \Compn{\Patl{H}}$. Moreover, the
  consequent $\forall I \in \Snl \colon \substring{x}{I} \in \rho_I$
  expresses the fact that $x$ coincides with some $\pi \in H$ on each
  $\ell$\hyp{}element subset; thus $x \preserves \rho_I$ for each
  $I$, i.e., $x \in \Aut \{\rho_I \mid I \in \Snl\}$. Thus the
  implication states that $\Compn{\Patl{H}} \subseteq \Aut \{\rho_I
  \mid I \in \Snl\}$.

  Assume now that $H = \Compn{S}$ for some $S \subseteq
  \symm{\ell}$. Condition~\eqref{CAAS:a} is necessary by
  Proposition~\ref{prop:l-closed}. With the properties of Galois connections (see Section~\ref{sec:3}),
  we have $H = \Compn{S} = \Compn{\Patl{\Compn{S}}} = \Compn{\Patl{H}}$, so condition~\eqref{CAAS:b} is also necessary.

  Assume then that conditions \eqref{CAAS:a} and \eqref{CAAS:b}
  hold. Since $\Compn{\Patl{}}$ is a closure operator, we have $H
  \subseteq \Compn{\Patl{H}}$. By conditions~\eqref{CAAS:a} and
  \eqref{CAAS:b}, we also have $\Compn{\Patl{H}} \subseteq \Aut
  \{\rho_I \mid I \in \Snl\} = H$. Thus $H = \Compn{\Patl{H}}$.
\end{proof}


\section{On the monotone Galois connection \newline $\boldsymbol{(\gPat, \gComp)}$}
\label{sec:5}

The monotone Galois connection $(\gPat, \gComp)$ is perhaps more interesting than $(\Pat, \Comp)$, because it makes a correspondence between well\hyp{}understood algebraic objects: permutation groups of two different degrees.
It may also be the more useful of the two when one wishes to investigate permutation classes in which every level is a group in greater detail than Atkinson and Beals did in \cite{AtkBea,AtkBea2001}.

As briefly mentioned in the previous section, the $\ell$\hyp{}pattern subgroups of $\symm{n}$ that are of the form $\Compn G$ for some subgroup $G$ of $\symm{\ell}$ can be described as automorphism groups of relations in a way that is much simpler than the one presented in Theorem~\ref{thm:Comp-Aut-arbitrary-S}.
In fact, as we will see in Theorem~\ref{thm:CompPat}, such subgroups are automorphism groups of a single relation of arity at most $\ell$.

Another goal of this section is to describe the Galois closures \linebreak $\gCompn \gPatl H$ and kernels $\gPatl \gCompn G$ as automorphism groups of relations that can be constructed from the group $H$ or $G$.
This is done in Theorems~\ref{thm:CompPat} and~\ref{thm:gPatgComp}.

In what follows, $\ell$ and $n$ are fixed integers satisfying $\ell \leq n$, and we will make use of the following constructions.
Let $k \leq \ell \leq n$, and let $\rho \subseteq \nset{n}^k$ and
$\sigma \subseteq \nset{\ell}^k$.
Define the relations $\CHECK{\rho} \subseteq \nset{\ell}^k$ and $\HAT{\sigma} \subseteq \nset{n}^k$ as
\begin{align*}
\CHECK{\rho} &:= \{h_I^{-1}(\vect{r}) \mid \vect{r} \in \rho, \, \range\vect{r} \subseteq I \in \Snl\}, \\
\HAT{\sigma} &:= \{h_J(\vect{s}) \mid \vect{s} \in \sigma, \, J \in \Snl\}.
\end{align*}
For $\sigma\subseteq\nset{\ell}^{\ell}_{\neq}$ we have
$\HAT\sigma=
  \{\vect{u} \in \nset{n}^\ell_{\neq} \mid \red(\vect{u}) \in \sigma\}$.

We are also going to consider the following condition for $\rho$:
\begin{equation}
\label{eq:condition-rho}
\forall \vect{r} \in \nset{n}^{k} \, \forall I, J \in \Snl: \range\vect{r} \subseteq
I \land \vect{r} \in \rho \implies h_J h_I^{-1}(\vect{r}) \in \rho.
\end{equation}

We remark that for $\rho\subseteq\nset{n}^{\ell}_{\neq}$
($k=\ell$) and with the notation from Section~\ref{sect:prelim} 
(cf.\ Lemma~\ref{A4}(ii)) this condition can be written as follows:
\begin{equation}
  \label{eq:condition}
  \forall \vect{r}, \vect{s} \in \nset{n}^\ell_{\neq} \colon \bigl( \vect{r} \in \rho \, \wedge \, \red(\vect{r}) = \red(\vect{s}) \bigr) \implies \vect{s} \in \rho,
\end{equation}
i.e., all tuples with a particular reduced form belong to $\rho$ whenever
one such tuple belongs to $\rho$.
A $k$\hyp{}ary $(k \leq \ell)$ relation satisfying condition
\eqref{eq:condition-rho} is called a \emph{pattern closed relation},
for short \emph{pc\hyp{}relation}. For $H \subseteq \symm{n}$ let
\[
  \pcInv H :=\{\rho \in \Inv H \mid \text{$\rho$ is a pc\hyp{}relation}\}
\]
denote the set of all invariant pc\hyp{}relations of $H$.

\begin{lemma}
\label{lem:condition-rho}
Let $k \leq \ell \leq n$ and $\rho \subseteq \nset{n}^k_{\neq}$.
Then $\rho \subseteq \HATCHECK{\rho}$.
Furthermore, $\rho = \HATCHECK{\rho}$ if and only if $\rho$ is a pc\hyp{}relation.
\end{lemma}

\begin{proof}
In order to prove $\rho \subseteq \HATCHECK{\rho}$, let $\vect{r} \in \rho$.
Let $I \in \Snl$ be any set satisfying $\range\vect{r}\subseteq I$.
Then $h_I^{-1}(\vect{r}) \in \CHECK{\rho}$ by the definition of $\CHECK{\rho}$, and consequently $\vect{r} = h_I \circ h_I^{-1}(\vect{r}) \in \HATCHECK{\rho}$.

Assume now that $\rho$ satisfies condition \eqref{eq:condition-rho}.
In order to show that $\HATCHECK{\rho} \subseteq \rho$, let $\vect{s} \in \HATCHECK{\rho}$.
Then $\vect{s} = h_J \circ h_I^{-1}(\vect{r})$ for some $I, J \in \Snl$ and $\vect{r} \in \rho$ satisfying $\range\vect{r} \subseteq I$.
By condition \eqref{eq:condition-rho}, we have $\vect{s} \in \rho$.

Assume then that $\rho = \HATCHECK{\rho}$.
Let $\vect{r} \in \rho$, $I, J \in \Snl$ such that $\range\vect{r} \subseteq I$.
Then $h_J h_I^{-1} (\vect{r}) \in \HATCHECK{\rho} = \rho$, and we conclude that $\rho$ satisfies condition \eqref{eq:condition-rho}.
\end{proof}

It turns out that every $\ell$\hyp{}ary relation of the form
$\HAT{\sigma}$ is a pc\hyp{}relation. This follows immediately from the following lemma.
\begin{lemma}
\label{lem:Ghatcheck=G}
Let $\sigma\subseteq\nset{\ell}^{\ell}_{\neq}$. Then $\CHECKHAT{\sigma} = \sigma$.
\end{lemma}

\begin{proof}
Recall $\HAT{\sigma} = \{h_J \circ \vect{s} \mid \vect{s} \in \sigma, J \in \Snl\}$.
Thus
\begin{align*}
\CHECKHAT{\sigma}
&= \{h^{-1}_I \circ \vect{r} \mid \vect{r} \in \HAT{\sigma}, \, \range\vect{r} \subseteq I \in \Snl\} \\
&= \{h^{-1}_{\range \vect{r}} \circ \vect{r} \mid \vect{r} \in \HAT{\sigma}\} 
= \{h^{-1}_J \circ h_J \circ \vect{s} \mid \vect{s} \in \sigma, \, J \in \Snl\} \\
&= \sigma.
\end{align*}
Concerning the second line in the equalities displayed above, note that the inclusion $\range \vect{r} \subseteq I$ holds as an equality, because $\vect{r}$ is an $\ell$\hyp{}tuple without repeated entries, i.e., $\card{\range \vect{r}} = \ell = \card{I}$; moreover, $\range (h_J \circ \vect{s}) = J$.
\end{proof}

The following theorem characterizes the closure operator
$\gCompn\gPatl$ of the monotone Galois connection $(\gPat,\gComp)$.

\begin{theorem}\label{thm:CompPat}
  \begin{itemize}
  \item[\rm(A)] $\gCompn\gPatl H=\Aut\pcInv H$ for $H\leq\symm{n}$.
  \item[\rm(B)] Let $H$ be a subgroup of $\symm{n}$. Then the
    following are equivalent:
    \begin{itemize}
    \item[\rm(a)] $H$ is Galois closed, i.e., $H=\gCompn\gPatl H$,
    \item[\rm(a)$'$] $\exists\, G\leq\symm{\ell}: H=\gCompn G$,
    \item[\rm(b)] $H=\Aut\pcInv H$,
    \item[\rm(c)] $\exists\, k\leq
      \ell\,\exists\,\rho\subseteq\nset{n}^{k}_{\neq}: 
      \rho=\HATCHECK{\rho}\land H=\Aut\rho$.
    \end{itemize}
  \end{itemize}
  
\end{theorem}

\begin{proof}
 (B): Clearly, (a)$\Leftrightarrow$(a)$'$, which follows from the properties
  of a (monotone) Galois connection.
  Also (c)$\Rightarrow$(b) is obvious, since
  $H\subseteq\Aut\pcInv H\subseteq\Aut\rho=H$.
  In order to prove (b)$\Rightarrow$(a), let $H=\Aut\pcInv
  H=\bigcap\{\Aut\rho \mid \rho\in \pcInv H\}$. In
  Proposition~\ref{prop:Autrho-CompAutrhocheck} below we shall see
  that $\Aut\rho$ is 
  Galois closed for each pc\hyp{}relation $\rho$ (i.e., if
  $\rho=\HATCHECK\rho$). Thus $H$ is the intersection of Galois closed
  groups and therefore also Galois closed. 
It remains to prove (a)$'${}$\Rightarrow$(c): Assume $H=\gCompn G$ for
$G\leq \symm{\ell}$. Then $\HAT{{\groupasrel{G}}}$ is a pc\hyp{}relation because
$\HAT{\CHECKHAT{{\groupasrel{G}}}} = \HAT{{\groupasrel{G}}} $ by
Lemma~\ref{lem:Ghatcheck=G} and we have $H=\Aut \HAT{{\groupasrel{G}}}$
by Proposition~\ref{prop:Autrho-CompAutrhocheck}.
(Recall the notation $\groupasrel{G}$ from Section~\ref{sec:4}.)

Finally, (A) follows from (B). In order to see this, let $H':=\Aut\pcInv H$ and
$H'':=\gCompn\gPatl H $. Then $H'$ is Galois closed 
by (b) (since $\pcInv\Aut\pcInv H=\pcInv H$, thus $\Aut\pcInv
H'=H'$) and contains $H$, consequently $H''\subseteq H'$ 
(because $\gCompn\gPatl$ is a closure operator).
Moreover, by (c) there exists a pc\hyp{}relation $\rho$ for the Galois
closure $H''$ such that 
$H''=\Aut\rho$. We have $\rho\in \pcInv H$
since $H\subseteq H''$. Therefore 
$H''=\Aut\rho\supseteq\Aut\pcInv H=H'$, consequently $H'=H''$.
\end{proof}

\begin{remark}
Theorem~\ref{thm:CompPat}(A) holds even for an arbitrary subset
$H\subseteq\symm{n}$ (since $\gPatl H=\gPatl\gensg{H}$ by Lemma~\ref{lem:R1},
and $\pcInv H=\pcInv\gensg H$).
\end{remark}

\begin{lemma}
\label{lem:Compsigma-Autsigmahat}
Let $k \leq \ell \leq n$, and let $\sigma \subseteq
\nset{\ell}^k_{\neq}$ be a $k$\hyp{}ary relation on $\nset{\ell}$.
Then we have $\gCompn \Aut \sigma \subseteq \Aut \HAT{\sigma}$.
\end{lemma}

\begin{proof}
Let $\pi \in \Compn \Aut \sigma$ and let $\vect{u} \in \HAT{\sigma}$.
Then $\vect{u} = h_J \circ \vect{s}$ for some $\vect{s} \in \sigma$ and $J \in \Snl$.
Since $\pi$ is compatible with $\Aut \sigma$, we have $\pi_J \in \Aut \sigma$, that is, $\pi_J(\vect{s}) \in \sigma$.
Consequently,
\[
\pi(\vect{u})
= \pi \circ h_J(\vect{s})
= h_{\pi(J)} \circ h_{\pi(J)}^{-1} \circ \pi \circ h_J (\vect{s})
= h_{\pi(J)} \circ \pi_J (\vect{s})
\in \HAT{\sigma},
\]
and we conclude that $\pi \in \Aut \HAT{\sigma}$.
\end{proof}

\begin{proposition}
\label{prop:Autrho-CompAutrhocheck}
Let $k \leq \ell \leq n$, and let $\rho \subseteq \nset{n}^k_{\neq}$ be a $k$\hyp{}ary pc\hyp{}relation on $\nset{n}$.
Then $\Aut\rho$ is Galois closed, more precisely 
$\Aut \rho = \gCompn \Aut\CHECK{\rho}$. 
In particular we have $\gCompn G = \Aut\HAT{{\groupasrel{G}}}$ for any $G \leq \symm{\ell}$.
\end{proposition}

\begin{proof}
Let $\pi \in \Aut \rho$. We show first that $\pi_J \preserves \CHECK{\rho}$ for all $J \in \Snl$.
Let $\vect{s} \in \CHECK{\rho}$.
Then $\vect{s} = h_I^{-1}(\vect{r})$ for some $\vect{r} \in \rho$ and $\range\vect{r} \subseteq I \in \Snl$.
It follows from condition~\eqref{eq:condition-rho} that $h_J h_I^{-1}(\vect{r}) \in \rho$.
Since $\pi \preserves \rho$, we also have $\pi h_J h_I^{-1}(\vect{r}) \in \rho$.
Consequently,
\[
\pi_J(\vect{s}) = h_{\pi(J)}^{-1} \circ \pi \circ h_J \circ h_I^{-1} (\vect{r}) \in \CHECK{\rho},
\]
and we conclude that $\pi \in \Compn \Aut \CHECK{\rho} = \gCompn \Aut \CHECK{\rho}$.
Therefore $\Aut \rho \subseteq \gCompn \Aut \CHECK{\rho}$.
The converse inclusion follows immediately from
Lemma~\ref{lem:Compsigma-Autsigmahat}:
$\gCompn \Aut \CHECK{\rho} \subseteq \Aut \HATCHECK{\rho} = \Aut \rho$
because  $\rho = \HATCHECK{\rho}$ by Lemma~\ref{lem:condition-rho}.

Finally, recall that any $G \leq \symm{\ell}$ can be considered as an $\ell$\hyp{}ary relation $\groupasrel{G} \subseteq \nset{\ell}^{\ell}$.
Thus $\rho := \HAT{{\groupasrel{G}}}$ is a pc\hyp{}relation and $\CHECK{\rho} = \groupasrel{G}$. It holds that $\Aut \groupasrel{G} = G$. Thus we get
$\Aut \HAT{{\groupasrel{G}}} = \Aut \rho = \gCompn \Aut \CHECK{\rho} = \gCompn \Aut \groupasrel{G} = \gCompn G$.
\end{proof}

Now we take a look at the other side of the monotone Galois connection
$(\gPatl, \gCompn)$, namely at the kernel operator $\gPatl\gCompn$. We
want to describe $G':=\gPatl\gCompn G$ as the automorphism group of some
relations. Clearly, because of $\Aut\Inv G' = G'\subseteq G = \Aut\Inv G$ we have to extend the set $\Inv G$ in order to get the
set $\Inv G'$. We shall see that the following definition will fit our purposes.

\begin{definition}\label{G3}
Let $G \leq \symm{\ell}$. A relation $\sigma \subseteq \nset{\ell}^{k}$
($k \leq \ell$) is called a \emph{pattern closed extended invariant \textup{(}pc\hyp{}extended invariant\textup{)} of $G$} if
$\CHECKHAT{\sigma} = \sigma$ and $\HAT{\sigma} \in \Inv\Aut \HAT{{\groupasrel{G}}}$. The
set of all pc\hyp{}extended invariants of $G$ is denoted by $\pcExt G$.
\end{definition}

\begin{remark}\label{G4}
Note that $\sigma \in \pcExt G$ implies $\HAT{\sigma} \in \pcInv \Aut \HAT{{\groupasrel{G}}}$.
Analogously to Lemma~\ref{lem:condition-rho} we have
$\sigma \subseteq \CHECKHAT{\sigma}$, and, furthermore, $\CHECKHAT{\sigma} = \sigma$ is
equivalent to the following condition:
\[
 \forall \vect{s} \in \nset{\ell}^{k} \, \forall I, J \in \Snl \colon
 \vect{s} \in \sigma \land \range{h_I(\vect s)} \subseteq J
 \implies h_I^{-1} h_J(\vect{s}) \in \sigma.
\]
\end{remark}

\begin{lemma}
\label{lem:Aut-rhocheck}
Let $H = \Aut \rho \leq \symm{n}$ for a pc\hyp{}relation $\rho \subseteq
\nset{n}^k$ where $k \leq \ell \leq n$. 
Then we have $\gPatl{H} \subseteq \Aut \CHECK{\rho}$.
\end{lemma}

\begin{proof}
Let $\sigma \in \Patl{H}$.
Then there exist $\tau \in H$ and $J \in \Snl$ such that $\sigma = \tau_J = h_{\tau(J)}^{-1} \circ \tau \circ h_J$.
Let $\vect{u} \in \CHECK{\rho}$.
Then $\vect{u} = h_I^{-1}(\vect{r})$ for some $\vect{r} \in \rho$ and $I \in \Snl$ satisfying $\range\vect{r} \subseteq I$.
Since $\rho$ satisfies condition \eqref{eq:condition-rho}, we have $h_J h_I^{-1}(\vect{r}) \in \rho$.
Since $\tau \preserves \rho$, we have $\tau h_J h_I^{-1}(\vect{r}) \in \rho$.
Consequently,
\[
\sigma(\vect{u})
= h_{\tau(J)}^{-1} \circ \tau \circ h_J \circ h_I^{-1}(\vect{r})
\in \CHECK{\rho}.
\]
This shows that $\Patl H \subseteq \Aut \CHECK{\rho}$.
Consequently, $\gPatl H \subseteq \Aut \CHECK{\rho}$.
\end{proof}

\begin{theorem}\label{thm:gPatgComp}
   $\gPatl\gCompn G=\Aut\pcExt G$ for $G\leq\symm{\ell}$.
\end{theorem}

\begin{proof}
Let $G_{1} := \gPatl\gCompn G$, $G_{2} := \Aut\pcExt G$.
Then $H := \gCompn G = \Aut \HAT{{\groupasrel{G}}} = \Aut\pcInv H$ by Theorem~\ref{thm:CompPat} and Proposition~\ref{prop:Autrho-CompAutrhocheck}.

We prove $G_{1} \subseteq G_{2}$. Let $\sigma \in \pcExt G$.
By Remark~\ref{G4}, we have $\HAT{\sigma} \in \pcInv \Aut\HAT{{\groupasrel{G}}} = \pcInv H$, which implies $\Aut\HAT{\sigma} \supseteq \Aut \pcInv H = H$.
Using Lemma~\ref{lem:Aut-rhocheck}, we get 
$G_{1} = \gPatl H \subseteq \gPatl \Aut\HAT{\sigma} \subseteq \Aut\CHECKHAT{\sigma} = \Aut \sigma$.
Consequently, $G_{1} \subseteq \bigcap \{\Aut \sigma \mid \sigma \in \pcExt G\} = G_{2}$.

For the converse implication $G_{2} \subseteq G_{1}$, consider $G_{1}$ as the $\ell$\hyp{}ary relation $\groupasrel{G_{1}} = G_1 \subseteq \nset{\ell}^{\ell}_{\neq}$.
Then $\groupasrel{G_{1}}$ is a pc\hyp{}extended invariant of $G$. Indeed, $\groupasrel{G_{1}} = \CHECKHAT{{\groupasrel{G_{1}}}}$ by Lemma~\ref{lem:Ghatcheck=G} and $\HAT{{\groupasrel{G_{1}}}}$ is a pc\hyp{}relation.
Moreover, from Proposition~\ref{prop:Autrho-CompAutrhocheck} and the properties of monotone Galois connections, we get $\Aut \HAT{{\groupasrel{G_{1}}}} = \gCompn G_{1} = \gCompn G = H$.
Then it follows that $\HAT{{\groupasrel{G_{1}}}} \in \pcInv H = \pcInv\Aut \HAT{{\groupasrel{G}}}$.
Consequently, $G_{2} = \Aut\pcExt G \subseteq \Aut \groupasrel{G_{1}} = G_{1}$.
\end{proof}


\section{Concluding remarks}
\label{sec:conclusion}

With the monotone Galois connection $(\Patl, \Compn)$ at hand, a natural question to ask is what its closed sets and kernels are.
In the current paper, we focused on those closed sets that are subgroups of $\symm{n}$.
A reasonable general description of all closed sets eludes us.

We would like to point out another direction to which our study inevitably leads.
Atkinson and Beals~\cite{AtkBea,AtkBea2001} studied \emph{group classes,} i.e., permutation classes in which all levels are groups.
In particular, they determined the possible asymptotical behaviours of level sequences $C^{(1)}, C^{(1)}, C^{(2)}, \dots$ of group classes $C$.
Furthermore, they fully described the group classes in which all levels are transitive groups.
As an attempt of refining Atkinson and Beals's results and looking deeper into the local behaviour of the level sequences of group classes, in~\cite{Lehtonen-patterns}, one of the current authors set about describing, for an arbitrary group $G \leq \symm{n}$, what the sequence
\[
\dots, \, \gPatl[n-2] G, \, \gPatl[n-1] G, \, G, \, \gCompn[n+1]{G}, \, \gCompn[n+2]{G}, \, \dots
\]
looks like.
The Galois connections $(\gPat, \gComp)$ and $(\Pat, \Comp)$ that were obtained in the current paper might be useful tools in further analysis of group classes.

We would also like to point out that the questions we are considering in this paper can be asked for other definitions of pattern involvement, e.g., consecutive, vincular, bivincular, mesh, etc.\ (see Kitaev~\cite{Kitaev}). This remains a topic of further investigation.


\section*{Acknowledgments}

The authors would like to thank Nik Ru\v{s}kuc for insightful and inspiring discussions.



\begin{thebibliography}{9}
\bibitem{AtkBea}
    \textsc{M. D. Atkinson,}  \textsc{R. Beals,}
    Permuting mechanisms and closed classes of permutations,
    in: C. S. Calude,  M. J. Dinneen (eds.),
    \textit{Combinatorics, Computation \& Logic,}
    Proc.\ DMTCS '99 and CATS '99 (Auckland),
    Aust.\ Comput.\ Sci.\ Commun., 21, No. 3, Springer, Singapore, 1999,
    pp.\ 117--127.

\bibitem{AtkBea2001}
    \textsc{M. D. Atkinson,}  \textsc{R. Beals,}
    Permutation involvement and groups,
    \textit{Q. J. Math.}\ \textbf{52} (2001) 415--421.

\bibitem{Bona}
    \textsc{M. B\'ona,}
    \textit{Combinatorics of Permutations,}
    Discrete Math.\ Appl.\ (Boca Raton),
    Chapman \& Hall/CRC, Boca Raton, 2004.

\bibitem{DenErnWis}
    \textsc{K. Denecke,}  \textsc{M. Ern\'e,}  \textsc{S.~L. Wismath}  (eds.),
    \textit{Galois Connections and Applications,}
    Math.\ Appl., vol.\ 565,
    Kluwer Academic Publishers, Dordrecht, 2004.

\bibitem{DixMor}
    \textsc{J. D. Dixon,}  \textsc{B. Mortimer,}
    \textit{Permutation Groups,}
    Grad.\ Texts in Math., vol.\ 163,
    Springer, New York, 1996.

\bibitem{Kitaev}
    \textsc{S. Kitaev,}
    \textit{Patterns in Permutations and Words,}
    Monogr.\ Theoret.\ Comput.\ Sci.\ EATCS Ser.,
    Springer, Heidelberg, 2011.

\bibitem{Lehtonen-patterns}
    \textsc{E. Lehtonen,}
    Permutation groups arising from pattern involvement,
    arXiv:1605.05571.

\bibitem{PosKal}
    \textsc{R. P\"oschel,}  \textsc{L. A. Kalu\v{z}nin,}
    \textit{Funktionen- und Relationenalgebren, Ein Kapitel der diskreten Mathematik,}
    VEB Deutscher Verlag der Wissenschaften, Berlin, 1979.

\end{thebibliography}
\end{document}